\documentclass[11pt]{article}  
\usepackage{amsmath}
\usepackage{amssymb}
\usepackage{theorem}
\usepackage{euscript}
\usepackage{pstricks}
\usepackage{authblk}
\usepackage{verbatim}
\usepackage{hyperref}

\usepackage{ifthen}
\usepackage{xifthen}
\usepackage{xcolor}
\usepackage{fullpage}
\usepackage[english]{babel}
\usepackage{mathtools}
\usepackage{hyperref}
\usepackage{color}

\usepackage{hyperref}
\hypersetup
{
	colorlinks,
	citecolor=black,
	filecolor=black,
	linkcolor=black,
	urlcolor=magenta
}
\hypersetup{linktocpage} 

\newtheorem{theorem}{Theorem}[section]
\newtheorem{lemma}[theorem]{Lemma}

\numberwithin{equation}{section}
\newtheorem{proposition}[theorem]{Proposition}

\newtheorem{remark}{Remark}[section]
\numberwithin{equation}{section}

\def\ZZ{{\mathbb Z}}

\def\RRd{{\mathbb R}^d}

\def\CC{{\mathbb C}}
\def\CC{{\mathbb C}}
\def\DD{{\mathbb D}}
\def\NN{{\mathbb N}}
\def\RR{{\mathbb R}}

\def\Vv{{\mathbb P}}

\def\Vv{{\mathcal P}}

\def\RRd{{\mathbb R}^d}

\def\RRi{{\mathbb R}^\infty}

\def\RRi{{\mathbb R}^\infty}

\def\RRp{{\mathbb R}_+}
\def\RRip{{\mathbb R}^{\infty}_+}

\def\Dd{{\mathcal D}}

\def\Ii{{\mathcal I}}
\def\Jj{{\mathcal J}}

\def\Vv{{\mathcal V}}

\def\CC{{\mathbb C}}
\def\ZZ{{\mathbb Z}}
\def\NN{{\mathbb N}}
\def\RR{{\mathbb R}}

\def\FF{{\mathbb F}}
\def\EE{{\mathbb E}}

\def\RRd{{\mathbb R}^d}

\def\supp{\operatorname{supp}}
\def\dv{\operatorname{div}}

\newcommand{\bb}{{\boldsymbol{b}}}

\newcommand{\be}{{\boldsymbol{e}}}

\newcommand{\bk}{{\boldsymbol{k}}}

\newcommand{\bp}{{\boldsymbol{p}}}

\newcommand{\bs}{{\boldsymbol{s}}}

\newcommand{\bx}{{\boldsymbol{x}}}

\newcommand{\by}{{\boldsymbol{y}}}

\newcommand{\brho}{{\boldsymbol{\rho}}}
\newcommand{\bsigma}{{\boldsymbol{\sigma}}}

\newcommand{\balpha}{{\boldsymbol{\alpha}}}

\newcommand{\bbeta}{{\boldsymbol{\beta}}}
\newcommand{\bdelta}{{\boldsymbol{\delta}}}
\newcommand{\bgamma}{{\boldsymbol{\gamma}}}
\newcommand{\blambda}{{\boldsymbol{\lambda}}}

\newcommand{\bvarepsilon}{{\boldsymbol{\bvarepsilon}}}

\newcommand{\rd}{{\rm d}}

\newcommand{\tL}{{\tilde{L}}}
\newcommand{\tv}{{\tilde{v}}}
\newcommand{\tu}{{\tilde{u}}}

\newcommand{\norm}[2]{\left\|{#1}\right\|_{#2}}

\newcommand{\brab}[1]{\left\{#1\right\}}
\newcommand{\brac}[1]{\left(#1\right)}

\title{\sffamily { Piecewise-polynomial interpolations  and quadratures for parametric PDEs with  log-Laplace  random inputs}}

\author[a]{Dinh D\~ung}
\affil[a]{Information Technology Institute, Vietnam National University, Hanoi
	\protect\\
	144 Xuan Thuy, Cau Giay, Hanoi, Vietnam
	\protect\\
	Email: dinhzung@gmail.com}

\date{\today}
\date{\ttfamily  \today}
\begin{document}
\maketitle

\begin{abstract}	
	We establish  a sparsity in terms of $\ell_p$-summability and weighted $\ell_2$-summability for the coefficients of the Laguerre generalized piecewise-polynomial chaos expansion of solutions to parametric elliptic PDEs with log-Laplace random inputs. From the sparsity, we derive convergence rates for semi-discrete approximations with respect to parametric variables. These rates are valid for sparse-grid, piecewise-polynomial interpolations and the generated  quadratures, and  to related extended least-squares approximations and generated quadratures.
	
	
\end{abstract}

\section{Introduction}
\label{Introduction}

 Let  $D \subset \RRd$ be a bounded  Lipschitz domain. We  consider the homogeneous divergence-form diffusion elliptic equation 
\begin{equation} \label{ellip}
	- \dv (a\nabla u)
	\ = \
	f \quad \text{in} \quad D,
	\quad u|_{\partial D} \ = \ 0. 
\end{equation}
Denote by $V:= H^1_0(D)$ the energy space and $V' = H^{-1}(D)$ the dual space of $V$. Throughout this paper, we assume that  $f \in V'$ and  $a \in L_\infty(D)$. Assume that  $a$ satisfies the ellipticity assumption
	\begin{equation} \nonumber
		0<\underset{\bx\in D}{\operatorname{ ess inf}}\,a(\bx) \leq a \leq  \underset{\bx\in D}{\operatorname{ ess sup}}\, a(\bx)<\infty.
	\end{equation}
By the well-known Lax-Milgram lemma, there exists a unique weak 
solution $u \in V$  to the equation~\eqref{ellip}  satisfying the variational formulation
\begin{equation} \nonumber
	\int_{D} a(\bx)\nabla u(\bx) \cdot \nabla v(\bx) \, \rd \bx
	\ = \
	\langle f , v \rangle,  \quad \forall v \in V.
\end{equation}
We are interested in approximation and numerical analysis for the parametric equation
\begin{equation} \label{SPDE}
	- {\rm div} (a(\by)\nabla u(\by))
	\ = \
	f \quad \text{in} \quad D,
	\quad u(\by)|_{\partial D} \ = \ 0,
\end{equation}	
with diffusion coefficient of the log-Laplace form
\begin{equation} \label{log-Laplace}
	a(\by)=\exp(b(\by)), \quad \ b(\by)=\sum_{j = 1}^\infty y_j\psi_j, \ \ \by \in \RRi,
\end{equation}
 where $\psi_j \in L_\infty(D)$ and $y_j$ are   i.i.d.  probability generalized Laplace random variables  associated with the  distribution  density 
$
 	l_a(y):=\frac{1}{2\Gamma(a)} e^{-|y|}|y|^{a-1},  \ a >0,
$
and $\Gamma$ is the gamma function.

The  problem of  semi-discrete parametric-variate approximation  in uncertainty quantification for  parametric PDEs with  random inputs  has been developed in
\cite{ADM2024,BCDM17,BCM17,BD2024,Dung21,DHN2026,MNST2014,NTW2008,NTW2008a,DNSZ2023,ZDS19,ZS20}.
The aim of the present paper is to study  the problem of linear semi-discrete {non-intrusive} approximation and numerical integration of the weak solution mapping $\by \mapsto u(\by)$ to the parametric PDE with log-Laplace random inputs  \eqref{SPDE}--\eqref{log-Laplace} and the corresponding convergence rates based on a finite number of particular solvers $u(\by_1),..., u(\by_n)$ for $\by_1,..., \by_n$ from $\RRi$. 
To the knowledge of the author, so far this problem for the  parametric equation with log-Laplace random inputs  \eqref{SPDE}--\eqref{log-Laplace} has not been considered in any prior works. It is worth emphasizing that the Laplace distribution is one of the fundamental probability distributions in both theoretical and applied contexts. Its distinctive shape and mathematical properties make it a key tool across a wide range of disciplines, from probability and statistical theories to practical data analysis and computational uncertainty quantification.

The key condition which emerges as governing the convergence rates of  approximation and numerical integration for solutions to parametric PDEs with random inputs is a sparsity of  the coefficients of the associated generalized polynomial chaos (GPC) expansion. Such a sparsity is quantified by $\ell_p$-summability  or weighted $\ell_2$-summability of these coefficients. 
We refer the reader to \cite[Section 3.9]{DNSZ2023} for a survey and bibliography on this sparsity problem.
Specially, the sparsity problem has been studied  for  parametric PDEs equations with  log-normal random inputs (for the Hermite GPC expansion) in \cite[Section 3.6]{DNSZ2023}, and with log-gamma random inputs (for the Laguerre GPC expansion) just recently in \cite{DHN2026}.
For the parametric PDE with log-Laplace random inputs  \eqref{SPDE}--\eqref{log-Laplace}, this problem meets with a certain obstruction. The techniques and arguments of complex-variable methods employed  
  which are  based particularly on the Rodrigues formula   \cite{ADMHolo,CCS13,CDS10,DNSZ2023,Zec18T,ZDS19} or on the   Laguerre's second order differential equation \cite{DHN2026} for the associated orthonormal polynomials, are not suitable for the log-Laplace context since so far we do not know such a formula  of  second order differential equation for the Laplace orthonormal polynomials. 
 The Laplace probability  distribution  has heavier tails than the normal distribution. While an expansion by the Hermite or Laguerre polynomials can be used as  for constructing corresponding GPC expansions which are a fundamental tool in uncertainty quantification for approximations of stochastic and parametric PDEs, finding a specific named sequence of orthogonal polynomials whose natural weight function (the distribution they are orthogonal with respect to) is the Laplace probability density function is not standard in the classical sense. In the present paper, we overcome this obstruction by proposing a new  tool for the log-Laplace context  to replace  the traditional GPC expansions. 
 

We give a short description of  the main contribution of the present paper. 
\begin{itemize}
	\item[{\rm (i)}] 
	We propose novel Laguerre generalized piecewise-polynomial chaos (GPWPC) expansions for the weak solution to parametric PDEs with log-Laplace random inputs \eqref{SPDE}--\eqref{log-Laplace} which play a basic role in  approximations of this solution (Section \ref{Sec:Laguerre GPWPC expansions}).
	
	\item[{\rm (ii)}]
	Based on these GPWPC expansions and certain bounds for  partial parametric derivatives  of the weak solution $u(\by)$ to \eqref{SPDE}--\eqref{log-Laplace}, we prove the sparsity 
	in terms of $\ell_p$-summability and weighted $\ell_2$-summability for the coefficients of the GPWPC expansion of  $u(\by)$ (Section~\ref{Sec:Sparsity for the Laguerre GPWPC expansion}, 
	Theorems~\ref{thm: summability-ell_p-rho-tL} and \ref{thm: summability-ell_p-b-tL}).
	
	\item[{\rm (iii)}] 
From this sparsity we derive convergence rates for the semi-discrete parametric-variate approximations of $u(\by)$ by sparse-grid piecewise-polynomial interpolations  and of the generated  quadratures as well  as of extended least squares approximations and generated quadratures (Section~\ref{Sec:Semi-discrete parametric approximations}, Theorems~\ref{thm:coll-approx-X}--\ref{thm: LS-sampling}).  
The convergence rate of extended least squares approximation which coincides with that of best $n$-approximation, crucially improves the convergence rate of sparse-grid piecewise-polynomial interpolation by a factor of $n^{-1/2}$ (comp. Theorems~\ref{thm:coll-approx-X} and \ref{thm: LS-sampling}).
\end{itemize}

The remaining part of the present paper is organized as follows. 
In Section~\ref{Sec:Laguerre GPWPC expansions}, we construct  Laguerre GPWPC expansions in Bochner spaces on $\RRi$  associated with a generalized Laplace probability measure.
In Section~\ref{Sec:Sparsity for the Laguerre GPWPC expansion},  we prove the sparsity 
in terms of $\ell_p$-summability and weighted $\ell_2$-summability for the coefficients of the GPWPC expansion of  the weak solution  to equation  \eqref{SPDE}--\eqref{log-Laplace}.
In Section~\ref{Sec:Semi-discrete parametric approximations}, from the sparsity results in the previous section, we  derive convergence rates for the semi-discrete  parametric-variable approximations of the weak solution  by truncated GPWPC expansion,  sparse-grid piecewise-polynomial interpolations and generated quadratures, extended least-squares sampling algorithms and the generated  quadratures. 

\medskip
\noindent
{\bf Notation} \ $\NN = \{1,2,3,\ldots\}$ denotes the natural numbers;
$\NN_0 = \NN \cup \{0\}$; $\RRp$ denotes the real non-negative numbers.
Denote by $\RR^\infty$ and $\RRip$ the
sets of all sequences $\by = (y_j)_{j\in \NN}$ with $y_j\in \RR$ and $y_j\in \RRp$, respectively.
Denote by $\FF$  the set of all sequences of non-negative integers $\bs=(s_j)_{j \in \NN}$ such that their support $J_\bs:= \{j \in \NN: s_j >0\}$ is a finite set.
We introduce in $\FF$ a half-ordering via 
$
\bs \leq \bs' \; \Longleftrightarrow \forall j\in \NN:\;\; s_j\leq s'_j.
$
A multi-indexed set
$(\sigma_\bs)_{\bs \in \FF} \subset \RR$ 
is called \emph{increasing} if 
$\sigma_{\bs'} \le \sigma_\bs$ for $\bs' \le \bs$. 
For $\bs = (s_j)_{j\in \NN} \in \FF$, 
we introduce 
$ |\bs|_1 := \sum_{j \in \NN} s_j
$; \ \ 
$
|\bs|_0 := \#\big(\{j\in \NN: s_j \ne 0 \}\big),\ \ 
|\bs|_\infty := \max_{j \in \NN} s_j 
$.
A multi-index set $\Lambda \subset \FF$ 
is called downward closed in $\FF$ 
if the inclusion $\bs \in \Lambda$ implies $\bs' \in \Lambda$ 
for every $\bs' \in \FF$  such that $\bs' \le \bs$.   
 For 
$\bs, \bk \in \FF$ and $k \in \NN$, 
 \ $\bs^k := \prod_{j \in \supp(\bs)}s_j^k$, 
$\bs^\bk := \prod_{j \in \supp(\bs)}s_j^{k_j}$, $\bs!:= \prod_{j \in \NN} s_j!$.
If  $\balpha= (\alpha_j)_{j \in \Jj}$  is a set of positive numbers with  any index set $\Jj$, then we use the notation 
$\balpha^{-1}:= (a_j^{-1})_{j \in \Jj}$. Denote $\balpha^\bbeta:= \brac{\alpha_j^{\beta_j}}_{j \in \Jj}$ for the sets $\balpha= (a_j)_{j \in \Jj}$ and $\bbeta= (\beta_j)_{j \in \Jj}$.
We use letters $C$, $K$ to denote general 
positive constants which may take different values, and $C_{\alpha,\beta...}$, $K_{\alpha,\beta,...}$ constants depending on $\alpha,\beta,...$.
Denote by $|G|$ the cardinality of the set $G$.

\section{Laguerre GPWPC expansions}
\label{Sec:Laguerre GPWPC expansions}

\subsection{Laguerre piecewise polynomials}
\label{Univariate Laguerre piecewise polynomials}

For a fixed number $a > 0$, recall, the generalized Laplace probability measure  $\lambda_a$ on $\RR$ is defined via the density function $l_a$ given  by
\begin{equation} \label{l_a(y)}
	l_a(y):=\frac{1}{2\Gamma(a)} e^{-|y|}|y|^{a-1}.
\end{equation}
The measure $\lambda_1$  with the density function $l_1(y):=\frac{1}{2} e^{-|y|}$  is the classical Laplace probability measure. In this section, we construct two different orthonormal piecewise-polynomial bases  for the Hilbert space $L_2(\RR;\lambda_a)$. The elements of these bases constitute the components of  GPWPC  which form bases of  the  Hilbert space  $L_2(\RRi;\blambda_a)$
equipped with infinite tensor-product measure $\blambda_a$ and which are a cornerstone  in GPWPC expansions for the associated Bochner space.

The measure $\lambda_a$ is related to  the gamma probability measure   $\gamma_a$ on non-negative semi-axis $\RRp$ which is defined via the density function
\begin{equation} \label{l(y)-Lag}
	g_a(y):=\frac{1}{\Gamma(a)} e^{-y}y^{a-1}.
\end{equation}
Let  $\Big(L_s^{(a-1)}\Big)_{s \in \NN_0}$ be the  Laguerre orthonormal polynomials with respect to $\gamma_a$.
For simplicity, we adopt the abbreviation:
$L_s:=L_s^{(a-1)}$. It is well-known that the  Laguerre polynomials $(L_s)_{s \in \NN_0}$ form an orthonormal basis of $L_2(\RRp;\gamma_a)$ (see,e.g., \cite{Szego1939}). 
 A  $\gamma_a$-measurable function $v$ on $\RR$ belongs to   $L_2(\RRp;\gamma_a)$ if and only if $v$ is represented by the series 
\begin{equation} \label{series-L}
	v =\sum_{s \in \NN_0} v_s \, L_s, \ \
\text{ where}
\ \
	v_s:= \int_{\RRp}	v(y)\,L_s(y) \rd \gamma_a(y),
\end{equation}
and  it holds  the Parseval's identity
\begin{equation} \label{Parseval's identity-L}
	\|v\|_{L_2(\RRp;\gamma_a)}^2
	\ = \ \sum_{s \in \NN_0}  |v_s|^2.
\end{equation}

Based on the  Laguerre polynomials, we construct   two orthonormal bases of $L_2(\RR;\lambda_a)$ which are constituted from Laguerre piecewise polynomials. 
Denote: $\RR_\delta:= \brab{y \in \RR: \delta y \ge 0}$ for $\delta=\pm 1$
(in particular, $\RR_{+1}=\RRp$).
Let  the double sequence $\brac{L_{\delta,s}}_{\delta=\pm 1,\, s \in \NN_0}$  be formed from Laguerre piecewise polynomials by
\begin{equation}  
	L_{\delta,s} (y):= \
	\begin{cases}
		2L_s(\delta y) \ \ &\text{if}\ \ y \in \RR_\delta, \\
		0\ \ &\text{if}\ \ \ y \not\in \RR_\delta.
	\end{cases}	
\end{equation}
Then $(L_{\delta,s})_{\delta=\pm 1,\, s \in \NN_0}$ is an orthonormal basis of the Hilbert space $L_2(\RR;\lambda_a)$, and each function $L_{\delta,s}$ is discontinuous at $0$. 
Moreover, a  $\lambda_a$-measurable function $v$ on $\RR$ belongs to   $L_2(\RR;\lambda_a)$ if and only if $v$ is represented by the ``double" series 
\begin{equation} \label{v=sum}
	v =\sum_{\delta=\pm 1,\, s \in \NN_0} v_{\delta, s} \,L_{\delta, s},\ \
	\text{ where}
	\ \
		v_{\delta, s}:= \int_{\RR}	v(y)\,L_{\delta, s}(y) \rd \lambda_a(y),
\end{equation}
and  it holds  the Parseval's identity
\begin{equation}\label{|v|=sum}
	\|v\|_{L_2(\RR;\lambda_a)}^2
	\ = \ \sum_{\delta=\pm 1,\, s \in \NN_0}  |v_{\delta, s}|^2,
\end{equation}
Indeed, if  $v\in L_2(\RR;\lambda_a)$, then   $v = \sum_{\delta=\pm 1}v_\delta$ with the equality in the sense of the space 
$L_2(\RR;\lambda_a)$, where 
\begin{equation}  
	v_\delta(y) = \
\begin{cases}
	v(\delta y) \ \ &\text{if}\ \ y \in \RR_\delta, \\
	0\ \ &\text{if}\ \ \ y \not\in \RR_\delta.
\end{cases}	
\end{equation}
Moreover,  the functions $v_{+1}$ and $v_{-1}$ are orthogonal in  the Hilbert space 
$L_2(\RR;\lambda_a)$, and the restrictions of $v_{+1}(y)$ and $v_{-1}(-y)$ to $\RRp$ belong to $L_2(\RRp;\gamma_a)$. Hence, we derive  that $(L_{\delta,s})_{\delta=\pm 1,\, s \in \NN_0}$ is an orthonormal basis of $L_2(\RR;\lambda_a)$ satisfying \eqref{v=sum}--\eqref{|v|=sum}.

Let  the sequence of functions $\big(\tL_s\big)_{s \in \NN_0}$ be defined by
\begin{equation}  
	\tL_s(y):= \
		2L_s(\delta y),\ \ y \in \RR_\delta.
\end{equation}
Observe that the functions $\tL_s$ are continuous piecewise polynomials on $\RR$ with a single non-smooth point at $0$, and
\begin{equation}  
	\tL_s (y)= L_{\delta,s} (y),  \ \ y \in \RR_\delta,
\end{equation}
and 
\begin{equation}  
	\tL_s = \sum_{\delta=\pm 1}L_{\delta,s}.
\end{equation}
Moreover,  the family $\big(\tL_s\big)_{s \in \NN_0}$ is an orthonormal basis of $L_2(\RR;\lambda_a)$. A  $\lambda_a$-measurable function $v$ on $\RR$ belongs to   $L_2(\RR;\lambda_a)$ if and only if $v$ is represented by the series 
\begin{equation} \label{series-tL}
	v =\sum_{s \in \NN_0} \tilde{v}_s \,\tL_s,\ \
	\text{ where}
	\ \
		\tilde{v}_s:= \int_{\RR}	v(y)\,\tL_s(y) \rd \lambda_a(y),
\end{equation}
and  it holds  the Parseval's identity
\begin{equation}
	\|v\|_{L_2(\RR;\lambda_a)}^2
	\ = \ \sum_{s \in \NN_0}  |\tilde{v}_s|^2.
\end{equation}
Moreover,
\begin{equation}
 |\tilde{v}_s|^2 
	\ = \ \sum_{\delta=\pm 1}  |v_{\delta, s}|^2 .
\end{equation}

We call the functions $\tL_s$ and $L_{\delta,s}$ Laguerre piecewise polynomials.

\subsection{Laguerre GPWPC expansions}
\label{Subsec:Laguerre GPWPC expansions}

The generalized Laplace 
probability measure $\blambda_a(\by)$ on $\RRi$ is defined as 
the infinite tensor product of one-dimensional generalized Laplace 
measures $\lambda_a(y_i)$:
\begin{equation} 
	\blambda_a(\by) 
	:= \ 
	\bigotimes_{j \in \NN} \lambda_a(y_j) , \quad \by = (y_j)_{j \in \NN} \in \RRi.
\end{equation}
(The sigma algebra for $\blambda_a(\by)$ is generated by the set of cylinders $A:= \prod_{j \in \NN} A_j$, where $A_j \subset \RR$ are univariate $\mu$-measurable sets and only a finite number of $A_i$ are different from $\RR$. For such a set $A$, we have $\blambda_a(A) = \prod_{j \in \NN} \lambda_a(A_j)$.) In the same manner, one can define the generalized gamma 
probability measure $\bgamma_a(\by)$ on $\RRip$ by replacing $\lambda(y_i)$ with $\gamma_a(y_i)$.

Let $X$ be a separable Hilbert space. For $0<p<\infty$, denote by $L_p(\RRi,X;\blambda_a)$ the Bochner space of strongly $\blambda_a$-measurable mappings $v$ from $\RRi$ to $X$, equipped with the norm
\begin{equation} \label{|v|_{L_2(RRip,X;lambda)}:=}
	\|v\|_{L_p(\RRi,X;\blambda_a)}
	:= \
	\left(\int_{\RRi} \|v(\by)\|_X^p \, \rd \blambda_a(\by) \right)^{1/p}.
\end{equation}
In a similar way, we can define the Bochner space  $L_p(\RRip,X;\bgamma_a)$ associated with infinite gamma probability measure $\bgamma_a$. 

Note that is a Hilbert space and there holds the identity (see, e.g., \cite{CSC2017})
\begin{equation} \label{L_2=}
L_2(\RRi,X;\blambda_a) \ = \ L_2(\RRi;\blambda_a) \otimes X.
\end{equation}

 Denote by  $\FF$  the set of all sequences of non-negative integers $\bs=(s_j)_{j \in \NN}$ such that their support 
$J_\bs:=J_\bs:= \{j \in \NN: s_j >0\}$ is a finite set. Let
$\DD:=\brab{\bdelta \in \FF : \ \delta_j = \pm 1, \ j \in J_\bdelta}$ and  
$\EE:=\brab{(\bdelta,\bs)\in \DD \times\FF: \,J_\bs = J_\bdelta}$.

For $\bs \in \FF$ and $(\bdelta,\bs)\in \EE$, we define the tensor products $\tL_\bs$ and $L_{\bdelta,\bs}$
of Laguerre piecewise polynomials as
\begin{equation*}
	\tL_\bs(\by)=\bigotimes_{j \in \NN}\tL_{s_j}(y_j), \ \
	L_{\bdelta,\bs}(\by)=\bigotimes_{j \in J_\bdelta}L_{\delta_j,s_j}(y_j).
\end{equation*}
Notice that $\brac{\tL_\bs}_{\bs \in \FF}$  and $(L_{\bdelta,\bs})_{(\bdelta,\bs)\in \EE}$ are  orthonormal bases for $L_2(\RRi;\blambda_a)$. Hence, by using \eqref{L_2=} we derive that 
a strongly $\blambda_a$-measurable $X$-valued function $v$ on $\RRi$ belongs to   $L_2(\RRi,X;\blambda_a)$ if and only if $v$ is represented   by one of the Laguerre  GPWPC expansions
\begin{equation} \label{GPWPCexpansion}
		v=\sum_{\bs \in \FF} \tv_\bs \,\tL_\bs, \ \  \tv_\bs \in X, \qquad
	v=\sum_{(\bdelta,\bs)\in \EE} v_{\bdelta,\bs}\,L_{\bdelta,\bs}, \quad v_\bs \in X,
\end{equation}
with equality and convergence in $L_2(\RRi,X;\blambda_a)$, where
\begin{equation*}
		\tv_\bs:=\int_{\RRi} v(\by)\, \tL_\bs(\by)\, \rd\blambda_a (\by), \ \ 
	\bs \in \FF, \qquad
	v_{\bdelta,\bs}:=\int_{\RRi} v(\by)\,L_{\bdelta,\bs}(\by)\, \rd\blambda_a (\by), \ \ 
	(\bdelta,\bs)\in \EE.
\end{equation*}
Moreover,  there hold  the Parseval's identities
\begin{equation} \label{Parseval's identities}
	\|v\|_{L_2(\RRi,X;\blambda_a)}^2
	\ = \ \sum_{\bs \in \FF}\|\tv_\bs\|_X^2, \qquad
	\|v\|_{L_2(\RRi,X;\blambda_a)}^2
	\ = \ \sum_{(\bdelta,\bs)\in \EE}\|v_{\bdelta,\bs}\|_X^2.
\end{equation}

%

Notice that these  GPWPC expansions with respect to the bases $\brac{\tL_\bs}_{\bs \in \FF}$  and $(L_{\bdelta,\bs})_{(\bdelta,\bs)\in \EE}$ are connected with the equalities
\begin{equation} \label{|tv_bs|=}
	\|\tv_\bs\|_X^2 = \sum_{\bdelta \in \DD: \, (\bdelta,\bs) \in \EE} \|v_{\bdelta,\bs}\|_X^2,
\end{equation}
and 
\begin{equation*}
	\tL_\bs = \sum_{\bdelta \in \DD: \, (\bdelta,\bs) \in \EE} L_{\bdelta,\bs},
\end{equation*}
with the equality in the Hilbert space $L_2(\RRi;\blambda_a)$.

A function $v \in L_2(\RRip,X;\bgamma_a)$ can be represented by the Laguerre GPC expansion
\begin{equation} \label{GPCexpansion-Lag}
	v(\by)=\sum_{\bs\in\FF} v_\bs \,L_\bs(\by), \quad v_\bs \in X,
\end{equation}
with convergence in $L_2(\RRip,X;\bgamma_a)$, where
\begin{equation*}
	L_\bs(\by)=\bigotimes_{j \in \NN}L_{s_j}(y_j),\quad 
	v_\bs:=\int_{\RRip} v(\by)\,L_\bs(\by)\, \rd\bgamma_a (\by), \quad 
	\bs \in \FF.
\end{equation*}
Notice that $(L_\bs)_{\bs \in \FF}$ is an orthonormal basis of $L_2(\RRip,\CC;\bgamma_a)$. 
Moreover, a strongly $\bgamma_a$-measurable function $v$ on $\RRip$ belongs to   $L_2(\RRip,X;\bgamma_a)$ if and only if $v$ is represented by the series \eqref{GPCexpansion-Lag} converging in $L_2(\RRip,X;\bgamma_a)$ and  it holds  the Parseval's identity
\begin{equation} \label{Parseval's identity}
	\|v\|_{L_2(\RRip,X;\bgamma_a)}^2
	\ = \ \sum_{\bs\in\FF} \|v_\bs\|_X^2.
\end{equation}

\section{Sparsity for  Laguerre GPWPC expansion}
\label{Sec:Sparsity for the Laguerre GPWPC expansion}

If
$b(\by)\in L_\infty(D)$ for $\by \in \RRi$, and $f\in V'$,
we have the inequality for the weak solution $u(\by)$:
\begin{equation} \label{eq:uApriori} \|u(\by)\|_V \leq \|f\|_{V'} \|
	a(\by)^{-1}\|_{L_\infty(D)} \leq \exp\big(\|b(\by)\|_{L_\infty(D)}\big)\|f\|_{V'}.
\end{equation}

We need the  following result proven in \cite[Lemma 3.9]{DNSZ2023}.
	Assume that for parametric equation \eqref{SPDE}--\eqref{log-Laplace}, there exist a positive sequence $\brho=(\rho_j)_{j\in \NN}$ and a number $\eta > 0$
	such that
	\begin{equation}\label{<kappa}
		\Bigg\|\sum_{j \in \NN} \rho_j |\psi_j | \Bigg\|_{L_\infty(D)} 
		\leq 
		\eta < \frac{\pi}{2}\,.
	\end{equation}
		Let $\by \in  \RRi$ with
		$b(\by)\in L_\infty(D)$, and $\bs\in \FF$ with
		$J_\bs\subseteq \supp(\brho)$.  Then  for the weak solution $u(\by)$, there hold the inequalities
		\begin{equation} \label{ineq-derivatives}
			\|\partial^{\bs}u(\by)\|_V 
			\leq
			e^\eta (\cos\eta)^{-1}\|f\|_{V'} \,
			\frac{\bs!}{\brho^\bs} 
			\exp\big( \|b(\by)\|_{L_\infty(D)} \big).
		\end{equation}

It is known that, the Laguerre polynomials $\big(L_s\big)_{s \in \NN_0}$ are nontrivial solutions to
the Laguerre's  differential equation of second order
\begin{equation} \label{Dd0} 
	\Dd v = k v, \quad \text{where} \quad \Dd   := - y  \frac{\rd^2  }{\rd y^2} - (a - y)\frac{\rd  }{\rd y},
\end{equation}
(see, e.g., \cite[5.1.3]{Szego1939}). This means that 
\begin{equation} \label{Dd1} 
	(\Dd   L_s)(y)   = s L_s(y),  \ \ y \in \RR, \ \ k \in \NN_0. 
\end{equation}
Hence it follows that  
\begin{equation} \label{Dd_delta} 
	(\Dd_\delta   L_{\delta,s})(y)   = s L_{\delta,s}(y),  \ \ 
	y \in \RR_\delta, \ \ k \in \NN_0, \ \ \delta=\pm 1, 
\end{equation}
where
\begin{equation} \label{Dd_delta1} 
	\Dd_\delta := - \delta y  \frac{\rd^2  }{\rd y^2} - \delta(a - \delta y)\frac{\rd  }{\rd y}.
\end{equation}
Note that, it holds the equality 
\begin{equation} \label{Dd2} 
	\Dd v = - e^y y^{1-a} \frac{\rd }{\rd  y} \brac{e^{-y} y^{a} \frac{\rd v}{\rd y}},
\end{equation}
and hence,
\begin{equation} \label{Dd_delta2} 
	(\Dd_\delta v)(y) = - 
	\delta e^{\delta y} (\delta y)^{1-a} 
	\frac{\rd }{\rd  y} \brac{e^{-(\delta y)} (\delta y)^{a} \frac{\rd v (y)}{\rd y}}, 
	\ \ y \in \RR_\delta, \ \ \delta=\pm 1.
\end{equation}

For  $\bs \in \FF$, denote: $ J_\bs:= \supp (\bs)$.
For a given  $\bdelta \in \DD$ and $r \in\NN$, we define the product differential operator 
 \begin{equation} \label{Dd^r_J} 
 	\Dd^r_\bdelta := \brac{\Dd_{\bdelta}}^r, \ \ 
 	\Dd_\bdelta  := \prod_{j \in J_\bdelta}\Dd_{\delta_j}:= (-1)^{|J_\bdelta|} \prod_{j \in J_\bdelta} \delta_j
 	e^{\delta_j y_j} (\delta_j y_j)^{1-a} 
 	\frac{\rd }{\rd  y_j} \brac{e^{-(\delta_j y_j)} (\delta_j y_j)^{a} \frac{\rd }{\rd y_j}}.
 	\end{equation}
Observe that there exist polynomials $\pi_{\delta,j}(t)$, $j = 1,...,2r$, of degree at most $r$  such that 
\begin{equation} \label{Dd^r} 
	\Dd_\delta^r 
	=
	\brac{ - \delta y  \frac{\rd^2 }{\rd y^2} - \delta(a - \delta y)\frac{\rd }{\rd y}}^r
	= \sum_{j=1}^{2r}	\pi_{\delta,j}(\delta y) \frac{\rd^j }{\rd y^j},
\end{equation}
and the coefficients of $\pi_{\delta,j}(y)$ depend on $r,a,\delta$ only.	Hence, there exists a constant $C_{a,r}$ such that for $y \ge 0$,
\begin{equation} \label{p_j} 
	|\pi_{\delta,j}(y)|  \le C_{a,r}(1 + |y|)^r, \ \ j = 1,...,2r.
\end{equation}

For  $\bdelta \in \DD$, denote: 
\begin{equation} \label{nu_J}
	\nu_{\bdelta ,\bs}:= \prod_{j \in J_\bdelta } s_j,
\end{equation}		
and  
$$
\RRi_{\bdelta}:= \prod_{j \in J_\bdelta} \RR_{\delta_j} \prod_{j \not\in J_\bdelta} \RR.
$$

\begin{lemma} \label{lemma:Dd^r_J v =Laguerre}
	Given $\bdelta \in \DD$, assume   $v \in L_2(\RRi,V;\blambda_a)$ and 
	$\Dd^r_\bdelta   v \in L_2(\RRi_{\bdelta},V;\blambda_a)$.  Then it holds the equality
 \begin{equation} \label{identity1} 
	(\Dd^r_\bdelta  v)(\by) = 
	\sum_{\bs\in \FF} \nu_{\bdelta ,\bs}^r v_{\bdelta,\bs} L_{\bdelta,\bs}(\by),\ \ 
	\text{a.e.}\ \by \in \RRi_{\bdelta}.
\end{equation}
\end{lemma}

\begin{proof}Let us  first  show that  if $v \in L_2(\RR,V;\lambda_a)$ and $\Dd_\delta v \in L_2(\RR_\delta,V;\lambda_a)$, then it holds for $\delta=\pm 1$,
\begin{equation} \label{identity2} 
	(\Dd_\delta v)(y) = \sum_{s\in \NN_0} s v_{\delta,s} L_{\delta,s}(y), \ \
		\text{a.e.}\ y \in \RR_\delta.
\end{equation}
This equality has been proven in \cite{DHN2026} for $\delta = 1$. For completeness, we give its  proof for this case.
If $\delta = 1$, taking account $\Dd_\delta = \Dd$,  $L_{\delta,s} = L_s$,  $v_{\delta,s} = v_s$ and $\RR_\delta=\RRp$, by  \eqref{Dd1}, \eqref{Dd2} and integration by parts twice, we have that,
\begin{equation} \label{eqs1}
	\begin{aligned}
		s  v_{\delta,s} &:= \ \int_{\RR} v(y)\, \brac{sL_{1,s}(y)}\,\rd\lambda_a(y)
		=  \ 	\int_{\RRp}v(y) \, \brac{sL_s(y)} \,\rd \gamma_a(y)\\[1ex]
		&=  \ 	\int_{\RRp}v(y) \Dd L_s(y) \,\rd y
		=  \ 	\int_{\RRp} \Dd v(y) L_s(y)\,\rd \gamma_a(y).
	\end{aligned}	
\end{equation}

Observe that  the Laguerre polynomials  $\big(L_s\big)_{s \in \NN_0}$ is an orthonormal basis of $L_2(\RRp;\lambda_a)$. 
By the Parseval's identity  with respect to this basis, from \eqref{eqs1}
we derive  
\begin{equation} \label{identity2a} 
	(\Dd v)(y) = \sum_{s\in \NN_0} s v_{\delta,s} L_s(y), \ \
	\text{a.e.}\ y \in \RRp,
\end{equation}
which coincides with \eqref{identity2} for $\delta =1$. 

If $\delta = -1$, by  changing variables: $\xi=-y$,  and applying \eqref{eqs1}, we obtain
\begin{equation} \label{eqs2}
	\begin{aligned}
		s  v_{\delta,s} &:= \ \int_{\RR} v(y)\, \brac{sL_{-1,s}(y)}\,\rd\lambda_a(y)
		=  \ 	\int_{\RR_{-}}v(y) \, \brac{sL_s(-y)} \,\rd \gamma_a(-y)
		\\[1ex]
		 &= \ \int_{\RRp} v(\xi)\, \brac{sL_s(\xi)}\,\rd\gamma_a(\xi)
		 	=  \ 	\int_{\RRp} \Dd  v(\xi)\,  L_s (\xi)\,\rd \gamma_a(\xi),		
	\end{aligned}
\end{equation}
where $u(\xi):= v(-\xi)$. As in \eqref{identity2a}, 
\begin{equation*} 
	(\Dd u)(\xi) = \sum_{s\in \NN_0} s v_{\delta,s} L_s(\xi), \ \
	\text{a.e.}\ \xi \in \RRp,
\end{equation*}
which is the same as \eqref{identity2} for $\delta = -1$. 
The equality \eqref{identity2} has been proven. Utilizing the tensor product argument, \eqref{identity2}, definition \eqref{Dd^r_J}, we derive \eqref{identity1}.
\hfill
\end{proof}

For $\theta, \lambda \ge 0$, we introduce the set 
$\bp(\theta, \lambda):= \brac{p_\bs(\theta, \lambda) }_{\bs \in \FF}$ by setting
\begin{equation} \label{[p_s]}
	p_\bs(\theta, \lambda) := \prod_{j \in J_\bs} (1 + \lambda s_j)^\theta, \quad \bs \in \FF.
\end{equation}
We will need the following inequality 
	\begin{equation}\label{ineq-theta,lambda}
		(1+ \lambda k)^{\theta} \le C_{\theta,\lambda} k^{\theta}, \ \  k \in \NN,
	\end{equation}
for fixed $\theta$ and $\lambda$. For  $0<p<\infty$ and $\theta \ge 0$, let  $r_{p,\theta}$ be chosen such that 
	\begin{equation}\label{C_{p,theta}}
p(r_{p,\theta}-\theta) > 1, \ \ \ \text{and \ denote} \ \ \ 
C_{p,\theta}:= \sum_{k \in \NN} k^{-p(r_{p,\theta}-\theta)}.
\end{equation}
For a finite subset $J$ of $\NN$ and $r \in \NN$, put
\begin{equation}\label{B_r(J):=}
	B_r(J):= \brac{\int_{\RRip}\prod_{j \in J} (1 + y_j)^{2r} 
		\exp\Big(2 \|b(\by)\|_{L_\infty(D)}\Big) \rd \blambda_a(\by)}^{1/2}.
\end{equation}

\begin{theorem} \label{thm: summability-ell_p-rho-tL}
	Assume that there exists  a positive sequence $\brho=(\rho_j)_{j\in \NN}$ satisfying  condition~\eqref{<kappa}, and $\brho^{-1}\in \ell_p(\NN)$ for some $0< p < \infty$. 
	Assume that  for  any  finite set $J\subset \NN$ and any $r\in \NN$, there exists a constant $K_r$ such that
	\begin{equation}\label{B_r(J)<}
		B_r(J) \le K_r^{|J|}.
	\end{equation}
		 
	Then $\brac{\norm{\tu_\bs}{V}}_{\bs\in \FF}\in \ell_{p}(\FF)$.	Moreover, if in addition, $p < 2$, for any fixed $\theta, \lambda \ge 0$,  we can construct a set 
	$\bsigma=(\sigma_\bs)_{\bs \in \FF}$  with positive $\sigma_\bs$, and a constant $M$ depending on $a,  p, \theta, \lambda,\eta$ such that
	\begin{equation} \label{ell_2-summability-rho}
		\left(\sum_{\bs\in\FF} (\sigma_\bs \|\tu_\bs\|_V)^2\right)^{1/2} \ \le M^{1/2} \ <\infty, \  \ \
		\norm{\bp(\theta,\lambda)\bsigma^{-1}}{\ell_{q}(\FF)} \le M^{1/q} < \infty,
	\end{equation}	
	where $q := 2p/(2-p)$. 
\end{theorem}
	
		\begin{proof} 
			We will utilize a technique developed in the proof of  \cite[Theorem 2.4]{DHN2026}, with certain modifications. 
				Let $\bdelta^* \in \DD$ be such that $\delta_j =1$ for all $j \in J_{\bdelta^*}$.
				Then	we have 
				$\Dd^r_{\bdelta^*}= \Dd^r$; \ $\RRi_{\bdelta^*} = \RRip$; \ 
				$\rd\blambda_a(\by)  = 2^{-|J_\bdelta^*|}\rd\bgamma_a(\by)$ and 
				$L_{\bdelta,\bs}(\by) = 2^{|J_\bdelta^*|} L_\bs(\by)$
				for $\by \in \RRi_{\bdelta^*}$; 
				\ and
				$u_{\bdelta,\bs} = u_\bs$.				
				Hence, by  Lemma~\ref{lemma:Dd^r_J v =Laguerre} 
				and the Parseval's  identity 
				for the Laguerre GPC expansion  \eqref{Parseval's identity},
					\begin{equation*}
					\int_{\RRip}\norm{\Dd^r_{\bdelta^*} u(\by)}{V}^2 \rd \bgamma_a(\by) 
					= \sum_{\bs \in \FF} \nu_{\bdelta^*,\bs}^{2r}\norm{u_\bs}{V}^2,
				\end{equation*}
or, equivalently,						
			\begin{equation*} 
				\int_{\RRi_{\bdelta^*}}\norm{\Dd^r_{\bdelta^*} u(\by)}{V}^2 \rd \blambda_a(\by) 
				= \sum_{\bs\in \FF} \nu_{\bdelta^*,\bs}^{2r} \norm{u_{\bdelta^*,\bs}}{V}^2.
			\end{equation*}	
For  any $\bdelta \in \DD$, by changing variables $\xi_j = \delta_j y_j$ for $j \in J_\bdelta$, and 
$\xi_j =  y_j$ for $j \not\in J_\bdelta$,	from the last equality we derive			
		\begin{equation} \label{P-identity} 
			\int_{\RRi_\bdelta}\norm{\Dd^r_\bdelta u(\by)}{V}^2 \rd \blambda_a(\by) 
			= \sum_{\bs\in \FF} \nu_{\bdelta ,\bs}^{2r} \norm{u_{\bdelta,\bs}}{V}^2.
		\end{equation}		
From \eqref{Dd^r} and \eqref{p_j}, we get for $\by \in \RRi_\bdelta$,	
		\begin{equation} \label{norm{Dd^r_J (u)}} 
			\begin{aligned}
			\norm{\Dd^r_\bdelta u(\by)}{V}
			&=
			\norm{\prod_{j \in J_\bdelta}\brac{ - y_j  \frac{\rd^2 }{\rd y_j^2} - (a - y_j)\frac{\rd }{\rd y_j}}^r u(\by)}{V}
			\\
			&\le 
		C_{a,r}^{|J_\bdelta|} \prod_{j \in J_\bdelta} (1 + |y_j|)^r 
			\sum_{\substack{\bk \in \FF: \, J_\bk=J_\bdelta, \\ \ |\bk|_\infty \le 2r}}\norm{\partial^\bk u(\by)}{V},
		\end{aligned}
		\end{equation}
		where $C_{a,r}$ is the constant as in \eqref{p_j}.
		By using \eqref{ineq-derivatives} we derive that
		\begin{equation} \label{norm{Dd^r_J (u)}} 
			\norm{\Dd^r_\bdelta u(\by)}{V}
			\le 
		C_{a,r}^{|J_\bdelta|} \prod_{j \in J_\bdelta} (1 + |y_j|)^r
			\sum_{\substack{\bk \in \FF: \, J_\bk=J_\bdelta, \\ \ |\bk|_\infty \le 2r}}C_0\frac{\bk!}{\brho^\bk}
			\exp\big( \|b(\by)\|_{L_\infty(D)} \big).
		\end{equation}
		Hence,	 we obtain
		\begin{equation}
			\int_{\RRi_\bdelta}\norm{\Dd^r_\bdelta u(\by)}{V}^2 \rd \blambda_a(\by) 
			\le 
		B_r(J_\bdelta)^2	C_0^2 	C_{a,r}^{2|J_\bdelta|} 
			\Bigg( \sum_{\substack{\bk \in \FF: \, J_\bk=J_\bdelta, \\ \ |\bk|_\infty \le 2r}}\frac{\bk!}{\brho^\bk}\Bigg)^2.
		\end{equation}
Moreover, $\bk! \le ((2r)!)^{|J_\bdelta|}$ when  $J_\bk = J_\bdelta, \ |\bk|_\infty \le 2r$.
	For $(\bdelta ,\bs) \in \EE$ with $\bs \in \FF$, we have  $J_\bs = J_\bdelta$ and
		$\nu_{\bdelta ,\bs}^{2r} = \bs^{2r}$.	Therefore, by \eqref{P-identity} and 
		\eqref{B_r(J)<}, 
\begin{equation} \label{norm{u_bs}{V}}
\begin{aligned}
				\norm{u_{\bdelta,\bs}}{V}	
	\le 
	\bs^{-r}	\Bigg(\int_{\RRi_\bdelta}\norm{\Dd^r_\bdelta u(\by)}{V}^2 \rd \blambda(\by) \Bigg)^{1/2}
\le 
C_0 C_1^{|J_\bdelta|}\bs^{-r}
\sum_{\substack{\bk \in \FF: \, J_\bk=J_\bdelta, \\ \ |\bk|_\infty \le 2r}}\brho^{-\bk},
\end{aligned}
\end{equation}
where $C_1:= K_rC_{a,r}(2r)!$ and $K_r$ is as in \eqref{B_r(J)<}. 

For a $\bs \in \FF$, the number of all 
$\bdelta \in \DD$ such that $(\bdelta, \bs)  \in \EE$ is $2^{|J_\bs|}$. Hence by 
\eqref{|tv_bs|=}, \eqref{norm{u_bs}{V}} and the equality $J_\bs = J_\bdelta$ we derive that
\begin{equation} \label{norm{u_bs}{V}2}
	\begin{aligned}
	\|\tu_\bs\|_V
		\le 
		C_0 (2C_1)^{|J_\bs|}\bs^{-r}
		\sum_{\substack{\bk \in \FF: \ J_\bk = J_\bs, \\ \ |\bk|_\infty \le 2r}} \brho^{-\bk}=: \beta_\bs,
	\end{aligned}
\end{equation}
Let $\bbeta:=(\beta_\bs)_{\bs \in \FF}$. For any $\theta',\lambda\geq 0$ we will prove that there exists $r \in \NN$ depending on $\theta'$ and $p$ only and a constant $M$ depending on 
$a,  p, \theta', \lambda,\eta$ only such that
\begin{equation} \label{norm1}
\norm{\bp(\theta',\lambda)\bbeta}{\ell_p(\FF)}
\ \le \
M^{1/p}.
\end{equation}
By \eqref{ineq-theta,lambda}  we have
\begin{equation*} 
	\begin{aligned}		
p_\bs(\theta',\lambda) \beta_\bs 	
\le 
C_0\bs^{-r+\theta'} (2C_{\theta',\blambda} C_1)^{|J_\bs|}
\sum_{\substack{\bk \in \FF: \, J_\bk= J_\bs, \\ \ |\bk|_\infty \le 2r}}\brho^{-\bk},
\ \ \bs \in \FF,
	\end{aligned}
\end{equation*}
where $C_{\theta',\lambda}$	 is as in 	\eqref{ineq-theta,lambda}. 
Hence, for a finite $J \subset \NN$,
\begin{equation} \label{sum2}
	\sum_{\bs \in \FF: \, J_\bs = J}	\big(p_\bs(\theta',\lambda) \beta_\bs\big)^p
	\leq 
	C_0^p(2C_{\theta',\lambda} C_1)^{p|J|}	\sum_{\bs \in \FF: \, J_\bs = J} \bs^{-p(r-\theta')} 
	\brac{\sum_{\substack{\bk \in \FF: \, J_\bk= J, \\ \ |\bk|_\infty \le 2r}} \brho^{-\bk}}^p.
\end{equation}
Let  $r= r_{p,\theta'}$ be chosen and $C_{p,\theta'}$ be defined as in \eqref{C_{p,theta}}.
Then we have 
\begin{equation} \label{sum2}
	\sum_{\bs \in \FF: \, J_\bs = J} \bs^{-p(r-\theta')} 
	=	
	\prod_{j \in J}	\sum_{s_j \in \NN} s_j^{-p(r-\theta')} 
	\le
	C_{p,\theta'}^{|J|}.
\end{equation}	
For $J:= \brab{j_1,\ldots,j_m}$ with $|J| := m$, we have that
\begin{equation*}
	\brac{\sum_{\substack{\bk \in \FF: \, J_\bk= J, \\ \ |\bk|_\infty \le 2r}} \brho^{-2\bk}}^p		
	= 
	\prod_{i=1}^m \brac{\sum_{\ell=1}^{2r}\brac{\rho_{j_i}}^\ell}^p.		
\end{equation*}
Hence,
\begin{equation*}
	\sum_{\bs \in \FF: \, J_\bs = J}	\big(p_\bs(\theta',\lambda) \beta_\bs\big)^p	
	\le 
	C_0^p(2C_{\theta',\lambda} C_1C_{p,\theta'})^m	\prod_{i=1}^m \brac{{\sum_{\ell=1}^{2r}}\brac{\rho_{j_i}}^\ell}^p,
\end{equation*}
or, equivalently,
\begin{equation}
	\sum_{\bs \in \FF: \, J_\bs = J}	\big(p_\bs(\theta',\lambda) \beta_\bs\big)^p	
	\le 
	C_0^p\prod_{i=1}^m K \brac{\sum_{\ell=1}^{2r}\brac{\rho_{j_i}}^\ell}^p,
\end{equation}
where 
$$
K=2C_{\theta',\lambda}K_rC_{a,r}C_{p,\theta'}(2r)!.
$$ It follows that
\begin{equation*}
	\begin{aligned}
		\sum_{\bs \in \FF}	\big(p_\bs(\theta',\lambda) \beta_\bs\big)^p	
		&
		\le C_0^p\sum_{m=1}^{\infty}\sum_{j_1,\ldots,j_m=1}^\infty 	
		\prod_{i=1}^m K \brac{\sum_{\ell=1}^{2r}\brac{\rho_{j_i}}^\ell}^p
		\\
		& =	C_0^p\prod_{k=1}^\infty 
		\Bigg(1+K \brac{\sum_{\ell=1}^{2r}\brac{\rho_{j_i}}^\ell}^p	\Bigg)
			\leq C_0^p
		\exp\Bigg(K\sum_{k=1}^\infty  
		\brac{\sum_{\ell=1}^{2r}\brac{\rho_{j_i}}^\ell}^p	\Bigg),
	\end{aligned}
\end{equation*}
which is finite since $\brho^{-1} \in \ell_p(\NN)$. This proves \eqref{norm1}.
Hence,  in particular, $\brac{\norm{\tu_\bs}{V}}_{\bs \in \FF}\in \ell_p(\FF)$ by \eqref{norm{u_bs}{V}2} .

If in addition, $0< p <2$, we define the set 
$\bsigma=(\sigma_\bs)_{\bs \in \FF}$ by
\begin{equation*} 
	\sigma_\bs:= 	\beta_\bs^{p/2 - 1},  \ \bs \in \FF. 
\end{equation*}	
Then
\begin{equation*}  \label{sum_ (bsigma^bs...e}
	\sum_{\bs\in\FF} (\sigma_\bs \|\tu_\bs\|_V)^2
	\ \le \
	\sum_{\bs\in\FF} 
	\brac{\beta_\bs^{p/2 - 1}\beta_\bs}^2
	\ = \
	\norm{\bbeta}{\ell_p(\FF)}^p 		
	\ \le \ M,
\end{equation*}	
and by choosing $\theta'=\theta q/p$,
\begin{equation*} 
	\big\|\bp(\theta,\lambda)\bsigma^{-1}\|_{\ell_{q}(\FF)}^{q}
	\ = \ 
	\sum_{\bs\in\FF} 
	\brac{\beta_\bs^{1- p/2}}^{p/(2-p)} p_\bs(\theta,\lambda)^{q}
	\ = \
	\norm{\bp(\theta',\lambda)\bbeta}{\ell_{q}(\FF)}^{q} 
	\ \le \
	M,
\end{equation*}	
which proves \eqref{ell_2-summability-rho}.		
		\hfill
	\end{proof}
\begin{theorem} \label{thm: summability-ell_p-b-tL}
	Let  $\bb = \brac{b_j}_{j \in \NN}$ be defined by	$b_j:= \norm{\psi_j}{L_\infty(D)}$. 		
	Assume that
	$\bb\in \ell_p(\NN)$ for some $0< p \le 1$, and 
	\begin{equation}\label{eq:max<1/2-1}
		\norm{\bb}{\ell_\infty(\NN)} = b_0 <\frac{1}{2},
\quad
		\norm{\bb}{\ell_1(\NN)}	< K^{-1},
	\end{equation}
	where 
	\begin{equation} \label{}
		K:= 2e (2r)!C_{a,r}C_{\theta,\lambda}C_{p,\theta}K_{a,r,\bb}, 
		\ \ r= r_{p,\theta}, \ \ \theta=0,
	\end{equation}
	and the constants $C_{a,r}$, $C_{\theta,\lambda}$, $C_{p,\theta}$, $K_{a,r,\bb}$, 
	$r_{p,\theta}$ are defined as in \eqref{p_j}, \eqref{ineq-theta,lambda}, \eqref{C_{p,theta}}, \eqref{K_{a,r,bb}}, \eqref{C_{p,theta}}, respectively.	
	Then $\brac{\norm{\tu_\bs}{V}}_{\bs\in \FF}\in \ell_{p}(\FF)$.	
	
	Moreover, if  $\theta'=2\theta /(2-p)$ and
	\begin{equation} \label{K}
		K:= 2e(2r)!C_{p,\theta'}C_{\theta',\lambda}K_{a,r,\bb}C_{a,r}, 
		\ \ r= r_{p,\theta'}
	\end{equation}
	for any fixed $\theta, \lambda \ge 0$,    we can construct a set 
		$\bsigma=(\sigma_\bs)_{\bs \in \FF}$  with positive $\sigma_\bs$, and a constant $M$ depending on $a,  p, \theta, \lambda,\eta$ such that there holds \eqref{ell_2-summability-rho}. 
\end{theorem}

\begin{proof} 	
		We will utilize a technique developed in the proof of  \cite[Theorem 2.6]{DHN2026}, with certain modifications. 
	We will need the inequality
	\begin{equation}\label{A_r(J)<}
		B_r(J) \le \exp \brac{\frac{a}{1-2b_0} \norm{\bb}{\ell_1(\NN)}}K_{a,r,\bb}^{|J|},
	\end{equation}
 which is proven in  \cite[Lemma  2.5]{DHN2026}, where
	\begin{equation} \label{K_{a,r,bb}}
		K_{a,r,\bb}:= \brac{\int_{\RR_+}  
			\frac{(1+y)^{2r} y^{a-1}}{\Gamma(a)} \exp\big(y(2b_0-1)\big) \rd y}^{1/2} <\infty.
	\end{equation}
	
	We define	the sequence $\brho_{\bk} = \brac{\rho_{\bk,j}}_{j \in \NN}$ depending on $\bk \in \FF$, by 
	\begin{equation}  
		\rho_{\bk,j}:= 
\begin{cases}
\frac{k_j}{b_j|\bk|_1} \ \ &\text{if}\ \ j\in J_\bk, \\
0\ \ &\text{if}\ \ j\not \in J_\bk.
\end{cases}		
\end{equation}		
Notice that
		\begin{equation}\label{<1}
	\sup_{\bk \in \FF}	\Bigg\|\sum_{j \in \NN} \rho_{\bk,j} |\psi_j | \Bigg\|_{L_\infty(D)} 
		\leq 1 < \frac{\pi}{2}.
	\end{equation}
	From \eqref{ineq-derivatives}  it implies that for $\by \in \RRi$  and $\bk\in \FF$,
	\begin{equation} \label{norm{Dd^r_J (u)}} 
\begin{aligned}
	\norm{\partial^\bk u(\by)}{V} &
\le 
C_0\frac{\bk!}{\brho_\bk^\bk}
\exp\big( \|b(\by)\|_{L_\infty(D)} \big)
=
C_0\frac{\bk!\bb^\bk |\bk|^{|\bk|}}{\bk^{\bk}}
\exp\big( \|b(\by)\|_{L_\infty(D)} \big).
\end{aligned}
	\end{equation}
	Similarly to the proof of the previous theorem we can derive for any 
	$\bdelta \in \DD$ and $r \in \NN$,
\begin{equation}
\begin{aligned}
\int_{\RRi_\bdelta}\norm{\Dd^r_\bdelta u(\by)}{V}^2 \rd \blambda_a(\by) 
&\le 
B_r(J_\bdelta)^2	C_0^2 	C_{a,r}^{2|J_\bdelta|} 
\Bigg( \sum_{\substack{\bk \in \FF: \, J_\bk=J_\bdelta, \\ \ |\bk|_\infty \le 2r}}
\frac{\bk!\bb^\bk |\bk|^{|\bk|}}{\bk^{\bk}}\Bigg)^2,
\end{aligned}
\end{equation}
where $C_{a,r}$ is the constant as in \eqref{p_j} and $B_r(\bdelta)$ as in \eqref{B_r(J):=}.
By \eqref{A_r(J)<} and the inequalities
$\frac{|\bk|^{|\bk|}}{\bk^{\bk}}\le \frac{|\bk|! e^{|\bk|}}{\bk!}$ and $\bk!\leq ((2r)!)^{|J_\bdelta|}$ with $J_\bk=J_\bdelta$ and $|\bk|_\infty \le 2r$,
we have
\begin{equation}
	\begin{aligned}
		\int_{\RRi_\delta}\norm{\Dd^r_\bdelta u(\by)}{V}^2 \rd \blambda(\by) 
		&\leq 	C_1^2K_{a,r,\bb}^{2|J_\bdelta|}		C_{a,r}^{2|J_\bdelta|} 
		((2r)!)^{2|J_\bdelta|}\Bigg( \sum_{\substack{\bk \in \FF: \, J_\bk=J_\bdelta, \\ \ |\bk|_\infty \le 2r}}\frac{(e\bb)^\bk |\bk|!}{\bk!}\Bigg)^2,
	\end{aligned}
\end{equation}
where
	\begin{equation} \label{}
C_1:= C_0 \exp \brac{\frac{a}{1-2b_0}\norm{\bb}{\ell_1(\NN)}}. 
\end{equation}
This together with  and \eqref{P-identity} implies
		\begin{equation} \label{norm{u_bs}2}
		\norm{u_{\bdelta,\bs}}{V}		
		\le 
	C_1 C_2^{|J_\bdelta|}	\bs^{-r} \sum_{\substack{\bk \in \FF: \, J_\bk=J_\bdelta, \\ \ |\bk|_\infty \le 2r}} \frac{(e\bb)^\bk |\bk|!}{\bk!}.
	\end{equation}
	where
	\begin{equation} \label{}
		C_2:= (2r)!K_{a,r,\bb}C_{a,r}. 
	\end{equation}
	Similarly to \eqref{norm{u_bs}{V}2}, from \eqref{norm{u_bs}2} it follows that
	\begin{equation} \label{norm{u_bs}{V}3}
		\begin{aligned}
			\|\tu_\bs\|_V
			\le 
			C_1 (2C_2)^{|J_\bs|}\bs^{-r}
			\sum_{\substack{\bk \in \FF: \ J_\bk = J_\bs, \\ \ |\bk|_\infty \le 2r}} 
			 \frac{(e\bb)^\bk |\bk|!}{\bk!}=: \beta_\bs,
		\end{aligned}
	\end{equation}
	
Let
$\bbeta=(\beta_{\bdelta,\bs})_{(\bdelta,\bs) \in \EE}$.
For any $\theta',\lambda\geq 0$ we will prove that there exists a constant 
$M$ depending on $a, \bb, p, \theta, \lambda$ only	 such that  
\begin{equation} \label{norm-1}
	\norm{\bp(\theta,\lambda)\bbeta}{\ell_p(\FF)}
	\ \le \
	M^{1/p}.
\end{equation}
By using the inequality \eqref{ineq-theta,lambda},  we get
\begin{equation*} \label{sum1}
	\sum_{\bs \in \FF: \, J_\bs = J}	\big(p_\bs(\theta',\lambda) \beta_\bs \big)^p
	\le 
	C_1^p(2C_{\theta',\lambda}C_2)^{p|J|} \sum_{\bs \in \FF: \, J_\bs = J}	\bs^{-p(r-\theta')}	
	\brac{\sum_{J_\bk = J, \ |\bk|_\infty \le 2r} \frac{(e\bb)^\bk |\bk|!}{\bk!}}^p,
\end{equation*}
where $C_{\theta',\lambda}$ is as in \eqref{ineq-theta,lambda}.
Let  $r= r_{p,\theta'}$ be chosen and $C_{p,\theta'}$ be defined as in \eqref{C_{p,theta}}. 
By \eqref{sum2} we get
\begin{equation*}
	\begin{aligned}
		\sum_{\bs \in \FF: \, J_\bs = J}	\big(p_\bs(\theta',\lambda) \beta_\bs \big)^p		
		&\le 
		C_1^pC_3^{p|J|}
		\sum_{J_\bk = J, \ |\bk|_\infty \le 2r} \brac{\frac{(e\bb)^\bk |\bk|!}{\bk!}}^p
		\\
		&\le  
		C_1^p\sum_{J_\bk = J, \ |\bk|_\infty \le 2r} 	\brac{\frac{(eC_3\bb)^\bk |\bk|!}{\bk!}}^p,		 
	\end{aligned}
\end{equation*}
where $C_3:=e^{-1}K$ and $K$ is as in \eqref{K}.
Then we obtain
\begin{equation} \label{sum3}
	\sum_{\bs\in \FF}	p_\bs(\theta',\lambda)^p \beta_\bs^p	
	\le  
	C_1^p\sum_{\bk \in \FF: \, |\bk|_\infty \le 2r} 
	\brac{\frac{\bar{\bb}^\bk |\bk|!}{\bk!}}^p, 
\end{equation}
where $\bar{\bb}:= K\bb$. 
We have by \eqref{norm-1} that $\norm{\bar{\bb}}{\ell_p(\NN)} < \infty$ and $\norm{\bar{\bb}}{\ell_1(\NN)} < 1$.
Hence, from \cite[Theorem 7.2]{CDS10} and \eqref{sum3} we derive the second inequality in \eqref{eq:max<1/2-1}, and therefore, 
$\norm{\bp(\theta,\lambda)\bbeta}{\ell_p(\FF)}\ \le \ M^{1/p}$.  It follows from \eqref{norm{u_bs}{V}3} that $\brac{\norm{u_\bs}{V}}_{\bs \in \FF}\in \ell_p(\FF)$ by choosing $\theta'=0$.

The  proof of  the bounds \eqref{ell_2-summability-rho} in Theorem~\ref{thm: summability-ell_p-b-tL} is similar to that of the same bounds  in 
Theorem~\ref{thm: summability-ell_p-rho-tL}.
	\hfill
\end{proof}

\section{Semi-discrete parametric  approximations}
\label{Sec:Semi-discrete parametric approximations}
In this section, from the weighted $\ell_2$-summability results of 
Theorems~\ref{thm: summability-ell_p-rho-tL} and \ref{thm: summability-ell_p-b-tL}, we  derive convergence rates for the semi-discrete linear parametric-variable approximations of the solution $u(\by)$ to the parametric equation with with log-Laplace random inputs  \eqref{SPDE}--\eqref{log-Laplace} by truncated GPWPC expansion,  sparse-grid piecewise-polynomial interpolations, extended least-squares sampling algorithms and the associated semi-discrete quadratures. 

\subsection{Approximation by finite truncated GPWPC expansion}
\label{Approximation by truncated GPWPC expansion}

It has been proven that under the assumptions of  Theorem~\ref{thm: summability-ell_p-rho-tL} or
Theorem \ref{thm: summability-ell_p-b-tL}
for any $0<p<2$ or $0<p\le 1$, respectively, and $\theta, \lambda \ge 0$,  we can construct a set 
$\bsigma=(\sigma_\bs)_{\bs \in \FF}$, and a constant $M$  such that	for the weak parametric solution $u$ to the equation \eqref{SPDE}--\eqref{log-Laplace},
\begin{equation} \label{ell_2-summability2}
	\left(\sum_{\bs\in\FF} (\sigma_\bs \|\tu_\bs\|_V)^2\right)^{1/2} \ \le M^{1/2} \ <\infty, \ \ \text{with} \ \
	\norm{\bp(\theta,\lambda)\bsigma^{-1}}{\ell_q(\FF)} \le M^{1/q} < \infty,
\end{equation}	
where $q := 2p/(2-p)$. Based on this weighted $\ell_2$-summability, we will construct linear approximation by by truncated GPWPC expansion.

If $0< p < 2$ and $q := 2p/(2-p)$, for $\xi>1$, we introduce the set 
\begin{equation} \label{Lambda(xi)}
	\Lambda(\xi)
	:= \ 
	\big\{\bs \in \FF: \, \sigma_{\bs}\le \xi^{1/q}  \big\}.
\end{equation}
Based on the first GPWPC expansion in \eqref{GPWPCexpansion}
 we define the truncation 
\begin{equation} \label{S_{Lambda(xi)}u}
	S_{\Lambda(\xi)} u 
	:= \ 
	\sum_{\bs\in {\Lambda(\xi)}} \tilde{u}_\bs \tL_\bs.
\end{equation}
Note that $S_{\Lambda(\xi)} u$ belongs the  linear subspace
$\Vv(\Lambda(\xi)):= \operatorname{span} \brab{\tL_\bs: \  \bs \in \Lambda(\xi)}$
 in $L_2(\RRi,V;\blambda_a)$.

\begin{proposition}\label{prop:L_2-approx}
	Under the assumptions of either Theorem \ref{thm: summability-ell_p-rho-tL} 
	or Theorem \ref{thm: summability-ell_p-b-tL},
	for every $n \in \NN$, we can find a number $\xi_n$ such that $\dim{V(\xi_n)} \le n$ and 
	\begin{equation} \nonumber
		\|u- S_{\Lambda(\xi_n)}u\|_{L_2(\RRi,V;\blambda_a)} 
		\leq C n^{- (1/p - 1/2)},
	\end{equation}
	where $C:= 2^{1/p-1/2}M^{1/p + 1/2}$.
\end{proposition}

\begin{proof}	
	By using the first Parseval's identity in \eqref{Parseval's identities}, from \eqref{ell_2-summability2}
	we derive for any $\xi >1$,
	\begin{equation} \label{u- S_{Lambda(xi)}u}
		\begin{split}
			\|u- S_{\Lambda(\xi)}u\|_{L_2(\RRi,V;\blambda_a)}^2 
			\ &= \
			\sum_{\sigma_{\bs}> \xi^{1/q} } \|\tilde{u}_\bs\|_V^2 
			\ = \
			\sum_{\sigma_{\bs}> \xi^{1/q} } (\sigma_{\bs}\|\tilde{u}_\bs\|_V)^2 \sigma_{\bs}^{-2}
			\\[1.5ex]
			\ &\le \
			\xi^{-2/q}
			\sum_{\bs \in \FF} (\sigma_{\bs}\|\tilde{u}_\bs\|_V)^2 
			\ = \ M^2 \xi^{-2/q}.
		\end{split}
	\end{equation} 
	On the other hand, again by \eqref{ell_2-summability2},
	\begin{equation} \nonumber
		\begin{split}
		\dim{V(\xi)} 	 
			\ = \
			\sum_{\sigma_{\bs}\le \xi^{1/q} } 1
			\ \le \
			\sum_{\sigma_{\bs}\le  \xi^{1/q} } \sigma_{\bs}^{-q}\xi
			\ \le \
			\xi
			\sum_{\bs \in \FF} \sigma_{\bs}^{-q}
			\ \le \ M\xi.
		\end{split}
	\end{equation} 	
Hence,	given any $n \in \NN$, by choosing a $\xi_n$ satisfying the condition
	$n/2 \ \le \ M\xi_n \ \le \ n$,
	we prove the lemma.
	\hfill
\end{proof}

\begin{remark} \label{remark4.1}
	{\rm
		The convergence rate of the linear approximation of $u$ by finite truncations of the Laguerre GPWPC expansion in Proposition~\ref{prop:L_2-approx} coincides with the   optimal convergence rate of best $n$-term Laguerre piecewise polynomial approximation of $u$,  which can be shown via  Stechkin's lemma. 
		(cf. \cite[Theorem 1.2]{BCDM17} and  \cite[Remark 4.1]{DNSZ2023} for  the   optimal convergence rate of best $n$-term Hermite polynomial approximation).
	}	
\end{remark}


\subsection{ Sparse-grid piecewise-polynomial  interpolation}
\label{ Hermite gpc interpolation approximation}
In this section, we construct sparse-grid piecewise-polynomial interpolations for semi-discrete parametric approximation of the solution $u(\by)$ to the parametric equation \eqref{SPDE}--\eqref{log-Laplace}.

From  \eqref{eq:uApriori} it follows that for
the parametric equation \eqref{SPDE}--\eqref{log-Laplace} admits a
unique solution for parameters $\by$ lying  in the set
\begin{equation}\label{U_0}
	U_0 := \big\{ \by \in \RRi: b(\by)\in L_\infty(D)\big\}.
\end{equation}

Suppose that	$\psi_j\in L_\infty(D)$,  $j\in \NN$, and there
exists a positive sequence $(\rho_j)_{j\in \NN}$ such that  the series 
$\big(\exp(-\rho_j)\big)_{j\in \NN}$ belongs to $\ell_1(\NN)$, and the
series $\sum_{j\in \NN}\rho_j|\psi_j|$ converges in the space
$L_\infty(D)$.
Then similar to the proof of  \cite[Theorem 2.2]{BCDM17}, we can show that the set $U_0$ has full measure, i.e.,
$
\blambda_a(U_0) = 1,
$
and	contains  all $\by \in \RRi$ with $|\by|_0 < \infty$.
Notice that $u(\by)$ is well-defined for every $\by\in U_0$. The set $U_0$	contains  all $\by \in \RRi$ with $|\by|_0 < \infty$, where $|\by|_0$ denotes the number of nonzero components $y_j$ of $\by$. Moreover, $u(\by)$ can be treated as a representative of an element in 
$L_2(\RRi,V;\blambda_a)$.

%
%
%

For $m \in \NN_0$, let $(y_{m;k})_{k=1}^m$ be the increasing sequence of  the $m$ roots of the Laguerre polynomial $L_m$, ordered as
$
0 < y_{m,1} < \cdots < y_{m,m}.
$
We extend this sequence as the symmetric sequence $(y_{m;k})_{|k|=1}^m$ to all $k$ with $0<|k|\le m$ by putting 
$y_{m;k}:=-y_{m;-k}$ for $k <0$.
We use also the convention $Y_0 = (y_{0;0})$ with $y_{0;0} = 0$. 

For  a function $v$  on $\RR$ taking values in a Hilbert space $V$ and $m \in \NN_0$, we define the   Lagrange piecewise-polynomial interpolation operator $I_m$  by
\begin{equation} \label{I_ v-Laplace}
	I_m v:= \ \sum_{|k|=1}^m v(y_{m;k}) \ell_{m;k}, 
\end{equation}
where for $k \in \NN$ with $1 \le |k|\le m$,
\begin{equation}  
	\ell_{m;k}(y):= 
	\begin{cases}
		\ell_{m;k}^+(y),   &\  y \ge 0, \\
		0, &  \  y <0,
	\end{cases}	
	\ \ \text{if} \ k>0,
	\quad
	\ell_{m;k}(y):= 
	\begin{cases}
		0,   &\  y >0, \\
		\ell_{m;-k}^+(-y),   &   \ y \le 0,
	\end{cases}	
	\ \ \text{if} \ k<0,
\end{equation}
and 
\begin{equation} \label{I_ v}
	\ell_{m;k}^+(y) := \prod_{\substack{j=1 \\  j\not=k}}^m\frac{y - y_{m;j}}{y_{m;k} - y_{m;j}},
\end{equation}
(in particular, $I_0 v = v(y_{0,0})\ell_{0,0}(y)= v(0)$ and $\ell_{0,0}(y)=1$). Notice that $I_m v$ is a function on $\RR$ taking values in $V$ and  interpolating $v$ at $y_{m;k}$ for $k$ with $1\le |k|\le m$, i.e., $I_m v(y_{m;k}) = v(y_{m;k})$.  The functions $\ell_{m;k}$  are piecewise polynomials with a single discontinuity point at $0$.

It is directly verify that
\begin{equation} \label{I_ v-Laplace}
	I_m v:= \ \sum_{\delta=\pm 1}I_{\delta,m} v , \quad
	I_{\delta,m} v:= \ \sum_{k=1}^m v(y_{m;\delta k}) \ell_{\delta,m;k}, 
	\end{equation}
where for $k =1,...,m$,
\begin{equation}  
	\ell_{\delta,m;k}(y):= 
	\begin{cases}
		\ell_{m;\delta k}^+(\delta y),   &\ \delta y \ge 0, \\
		0, &  \  \delta y <0,
	\end{cases}	
\end{equation}

Let
\begin{equation*} \label{lambda_m}
	\lambda_m:= \ \sup_{\big\|v\sqrt{l_a}\big\|_{L_\infty(\RR)} \le 1} 
	\big\|(I_m v)\sqrt{l_a}\big\|_{L_\infty(\RR)} 
\end{equation*}	
be the Lebesgue constant, where $l_a$ is given as in \eqref{l_a(y)}.
It was proven in \cite{MO2001} that
\begin{equation} \nonumber
	\lambda_m
	\ \approx \
	Cm^{1/6}, \quad m \in \NN,
\end{equation}
for some positive constant $C$ independent of $m$ (with the obvious inequality $\lambda_0(Y_0) \le 1$). Hence, for every $\varepsilon > 0$, there exists a positive constant $C_\varepsilon \ge 1$ independent of $m$ such that
\begin{equation*} \label{ineq[lambda_m]2}
	\lambda_m
	\ \le \
	(1 + C_\varepsilon m)^{1/6 + \varepsilon}, \quad \forall m \in \NN_0.
\end{equation*}


%
%

For  a function $v$  on $\RRi$ taking values in a Hilbert space $V$ and $\bs \in \FF$, 
we introduce the tensor product operators $\Delta_\bs$, $\bs \in \FF$ and 
$\Delta_{\delta,\bs}$, $(\delta,\bs) \in \EE$ , by
\begin{equation*} \label{Delta_bs v}
		I_\bs v
	:= \
	\bigotimes_{j \in \NN} I_{s_j} v; \quad
	\Delta_\bs v
	:= \
	\bigotimes_{j \in \NN} \Delta_{s_j} v, \quad 
	\Delta_m
	:= \
	I_m - I_{m-1}, \ \ I_{-1} = 0 \ \ (m \in \NN_0),
\end{equation*}
where the univariate operators $I_{s_j}$ and
$\Delta_{s_j}$ are successively applied to the univariate functions $\bigotimes_{i<j} \Delta_{s_i} v$ and $\bigotimes_{i<j} I_{s_i} v$, respectively,  by considering them as 
functions of  variable $y_j$ with the other variables held fixed.

For a given finite set $\Lambda \subset \FF$, we introduce the polynomial  interpolation operator $I_\Lambda$  by
\begin{equation*} \label{I_Lambda}
	I_\Lambda
	:= \
	\sum_{\bs \in \Lambda} \Delta_\bs.
\end{equation*}



Let $0 < q < \infty$ and $\bsigma= (\sigma_\bs)_{\bs \in \FF}$ be a set of positive numbers.
If $\Lambda(\xi)$ is the set defined as in \eqref{Lambda(xi)},  we  can  represent the interpolation operator $I_{\Lambda(\xi)}$ in the form
\begin{equation*} \label{I_Lambda(xi)=}
	I_{\Lambda(\xi)} v			
	\ = \
	\sum_{(\bs,\be,\bk) \in G(\xi)}  (-1)^{|\be|_1} v(\by_{\bs - \be;\bk})\ell_{\bs - \be;\bk},
\end{equation*}
where	
\begin{equation} \nonumber
	\ell_{\bs;\bk}(\by)
	:= \
	\prod_{j \in \NN} \ell_{s_j;k_j}(y_j), 
\end{equation}
\begin{equation*} \label{G(xi):=}
	G(\xi)				
	:= \
	\{(\bs,\be,\bk): 
	\, \bs \in \Lambda(\xi), \ \be \in E_\bs, \ \bk \in P_{\bs,\be}\},
\end{equation*}
for $\bs \in \FF$,  $E_\bs$ is the subset in $\FF$ of all $\be$ such that $e_j$ is either $1$ or $0$ if $s_j > 0$, and $e_j$ is $0$ if $s_j = 0$, $|\bs|_1 := \sum_{j \in \NN} s_j$, 
$\by_{\bs;\bk}:= (y_{s_j;k_j})_{j \in \NN} \in \RRi$ and 
$$
P_{\bs,\be}:= \brab{\bk =
	\brac{k_j}_{j \in \NN}: k_j \in \ZZ,\, 1\le |k_j|\le s_j - e_j \ \text{if}\ j \in J_\bs, \ \text{and} \  k_j =0 \ \text{if}\ j \not\in  J_\bs}.
$$
	
Note $I_{\Lambda(\xi)} v$ is determined by the values of $v$ at the points $\by_{\bs - \be;\bk}$, 
$(\bs,\be,\bk) \in G(\xi)$, and the number of these points is $|G(\xi)|$. Moreover, 
$|\by_{\bs - \be;\bk}|_0 < \infty$, and, consequently  $\by_{\bs - \be;\bk} \in U_0$ for every $(\bs,\be,\bk) \in G(\xi)$. Hence, for the weak solution $u$ to \eqref{SPDE}--\eqref{log-Laplace}, the function $I_{\Lambda(\xi)} u$ is well-defined.

\begin{theorem}\label{thm:coll-approx-X}
Let the assumptions of either Theorem~\ref{thm: summability-ell_p-rho-tL} 
	or Theorem \ref{thm: summability-ell_p-b-tL} for  some $0<p<1$. 		Let $\Lambda(\xi)$ be the set defined in \eqref{Lambda(xi)} for the set $\bsigma$  in one of these theorems satisfying \eqref{ell_2-summability2}. Then there exists a constant $C$ such that for each $n > 1$, we can construct a sequence of points  
	$(\by_{\bs - \be;\bk})_{(\bs,\be,\bk) \in G(\xi_n)}$ so that
	$|G(\xi_n)| \le n$ and
	\begin{equation*} \label{u-I_Lambda u}
		\|u -I_{\Lambda(\xi_n)}u\|_{L_2(\RRi,V;\blambda_a)} 
		\leq C n^{-(1/p - 1)}.
	\end{equation*}
\end{theorem}

\begin{proof}		
	We have by  the H\"older inequality, the inequality $q <2$ and \eqref{ell_2-summability2}, 
	\begin{equation} \nonumber
		\begin{split}
			\sum_{\bs\in\FF} \|\tilde{u}_\bs \tL_\bs\|_{L_2(\RRi,V;\blambda_a)}
			\ &= \
			\sum_{\bs\in\FF} \|\tilde{u}_\bs\|_V \|\tL_\bs\|_{L_2(\RRi;\blambda_a)}
			\\[1ex]
			\ &\le \
			\sum_{\bs\in\FF} \|\tilde{u}_\bs\|_{X} 
			\ \le \
			\left(\sum_{\bs\in\FF} (\sigma_{\bs}\|\tilde{u}_\bs\|_{X} )^2\right)^{1/2}
			\left(\sum_{\bs\in\FF} \brac{\sigma_{\bs}^{-1}}^2\right)^{1/2}
		\\[1ex]
		\ &\le \
		\left(\sum_{\bs\in\FF} (\sigma_{\bs}\|\tilde{u}_\bs\|_{X} )^2\right)^{1/2}
		\left(\sum_{\bs\in\FF} \brac{\sigma_{\bs}^{-1}}^q\right)^{1/q}
			\ < \ \infty.
		\end{split}
	\end{equation}
	Hence, the series \eqref{series-tL} for $v=u$
	converges absolutely, and therefore, unconditionally  in  the Hilbert space $L_2(\RRi,V;\blambda_a)$ to $u$. 
	
	In the next step, we prove that	 
	\begin{equation} \label{Delta_bs}
	\Delta_\bs \tL_{\bs'} = 0  \ \ \text{if} \ \	\bs, \bs' \in \FF, \ \bs \not\le \bs'.
	\end{equation}
By the tensor argument, it 	is sufficient to prove that 
$\Delta_s \tL_{s'} = 0$ if $s,s' \in \NN_0$ and $s > s'$. Indeed,  since the operator $I_m^+$ preserves every polynomial of degree at most $m$, and $s > s'$, from the definitions it is easy to check that 
$I_{\delta,s}L_{\delta,s'} = I_{\delta,s-1}L_{\delta,s'}=L_{\delta,s'}$. On the other hand, obviously, 
$I_{+1,s}L_{-1,s'} = I_{-1,s}L_{+1,s'}=0$.
Hence,
	\begin{equation*} \label{Delta_bs v}
\Delta_s\tL_{s'}:=		\brac{I_s - I_{s-1}}\tL_{s'} 
= 
\sum_{\delta=\pm1} \brac{I_{\delta,s}L_{\delta,s'} - I_{\delta,s-1}L_{\delta,s'}}
= 0.
	\end{equation*}
The relation \eqref{Delta_bs} has been proven.	This relation yields that
$I_{\Lambda(\xi)} \tL_\bs= \tL_\bs$ for every $\bs \in \Lambda(\xi)$, 
because $\Lambda(\xi)$ is a downward closed set in $\FF$. Therefore, by using the unconditional convergence of the series \eqref{series-tL}, we derive that
	\begin{equation}\label{I_Lambda}
		I_{\Lambda(\xi)}u
		\ = \
		I_{\Lambda(\xi)}\Big(\sum_{ \bs \in \FF} \tilde{u}_\bs \,\tL_\bs\Big)
		\ =  \
		\sum_{ \bs \in \FF} \tilde{u}_\bs \,I_{\Lambda(\xi)} \tL_\bs
		\ =  \
		\sum_{ \bs \in \Lambda(\xi)}  \tilde{u}_\bs \, \tL_\bs
		\ + \ \sum_{\bs \not\in \Lambda(\xi)} \tilde{u}_\bs \, I_{\Lambda(\xi) \cap R_\bs}\, \tL_\bs.
	\end{equation}	
	This implies that
	\begin{equation} \label{u-I_Lambda}
		u  - I_{\Lambda (\xi)} u
		\ = \
	u  -  S_{\Lambda (\xi)} u
		- 
		\ \sum_{\bs \not\in \Lambda(\xi)} \tilde{u}_\bs \, I_{\Lambda(\xi) \cap R_\bs}\, \tL_\bs,
	\end{equation}
	where 
		\begin{equation} \label{[|u-Iu|<]1}
		\big\|u  -  I_{\Lambda (\xi)} u\big\|_{L_2(\RRi,V;\blambda_a)}
		 \le \
		\big\|u- S_{\Lambda (\xi)} u\big\|_{L_2(\RRi,V;\blambda_a)}
		+  
		\ \sum_{\bs \not\in \Lambda(\xi)} \|\tilde{u}_\bs \|_V \, 
		\big\|I_{\Lambda(\xi) \cap R_\bs}\, \tL_\bs\big\|_{L_2(\RRi;\blambda_a)}.
	\end{equation}	
In a fashion analogous to the proof of \cite[(3.28)]{Dung21}	one can  find numbers $\theta$ and $\lambda$ so that
	\begin{equation} \label{I_Lambda}
		\big\|I_{\Lambda(\xi) \cap R_\bs}\, \tL_\bs\big\|_{L_2(\RRi;\blambda_a)}
		\ \le \
		p_\bs(\theta, \lambda). 
	\end{equation}
Denoting by $B(\xi)$ the sum in the right-hand of  \eqref{[|u-Iu|<]1},	we get by \eqref{I_Lambda}, the H\"older inequality and \eqref{ell_2-summability2},
	\begin{equation*} 
		\begin{split}
		B(\xi)	
		\ &\le \ \sum_{\sigma_\bs^{q} > \xi}   \|\tilde{u}_\bs\|_V\, p_\bs(\theta,\lambda)
			\ \le \
			\left(\sum_{\sigma_\bs^{q} > \xi} (\sigma_\bs\|\tilde{u}_\bs\|_V)^2\right)^{1/2}
			\left(\sum_{\sigma_\bs^{q} > \xi} p_\bs^2(\theta,\lambda) \sigma_\bs^{-2}\right)^{1/2}
			\\[1.5ex]
			\ &\le \
			C\,
			\left(\sum_{\sigma_\bs^{q} > \xi} p_\bs^2(\theta,\lambda) \sigma_\bs^{-q} 
			\sigma_\bs^{-(2- q)}\right)^{1/2}
			\ \le \
			C (\xi)^{-(1/q - 1/2)}
			\left(\sum_{\bs \in \FF} p_\bs(2\theta,\lambda) \sigma_\bs^{-q} \right)^{1/2}
			\\[1.5ex]
			\ &\le \
			C \xi^{-(1/q - 1/2)} \ = \ C \xi^{-(1/p - 1)}.
		\end{split}
	\end{equation*} 
This together with 	\eqref{[|u-Iu|<]1} and \eqref{u- S_{Lambda(xi)}u} gives the bound
	\begin{equation} \label{u-I_Lambda u}
	\|u -I_{\Lambda(\xi)}u\|_{L_2(\RRi,V;\blambda_a)} 
	\leq C \xi^{-(1/p - 1)}.
\end{equation}
Since the number of sample points in $|\Delta_\bs|$ does not exceed 
$\prod_{j \in \NN} (2s_j + 1) = p_\bs(1,2)$, we have from the definition
\begin{equation} \label{Gamma(Lambda(xi))}
	|G(\xi)|
	\ \le  \ 
	\sum_{\sigma_\bs \le \xi^{1/q}} p_\bs(1,2) 
	\ \le \ K \xi,
\end{equation}
where $K:= \sum_{\bs \in \FF} p_\bs(1,2)  \sigma_\bs^{-q}  < \infty$ by \eqref{ell_2-summability2}. For any $n \in \NN$, by choosing a number $\xi_n$ satisfying the inequalities 
$K \xi_n \le n <  2K \xi_n$, from \eqref{u-I_Lambda u} 
we  derive the theorem.
	\hfill
\end{proof}

\subsection{Sparse-grid quadrature}

If $v$ is a function defined on $\RR$ taking values in a Hilbert space $X$,  the function  $I_m v$ in  \eqref{I_ v}  generates the quadrature formula which is defined by
\begin{equation} \nonumber
	Q_m v
	:= \ \int_{\RR} I_m v(y) \, \rd \lambda_a(y)
	\ = \
	\sum_{k=0}^m\omega_{m;k}\, v(y_{m;k}), \quad
	\omega_{m;k}
	:=  \int_{\RR} \ell_{m;k}(y) \, \rd \lambda_a(y).
\end{equation}
 We define the univariate operator $\Delta^{{\rm Q}}_m$ for even $m \in \NN_0$ by
\begin{equation} \nonumber
	\Delta^{{\rm Q}}_m
	:= \
	Q_m - Q_{m-1},
\end{equation} 
with the convention $Q_{-1} := 0$. The quadrature operators $\Delta^{{\rm Q}}_\bs$ for $\bs \in \FF$, $Q_\Lambda$ for a finite set $\Lambda \subset \FF$, are defined in a similar way  by  replacing $\Delta_{s_j}$ with $\Delta^{{\rm Q}}_{s_j}$, $j \in \NN_0$.

The quadrature operator $Q_\Lambda$ is generated from the interpolation operator $I_\Lambda$ as 
\begin{equation} \nonumber
	Q_\Lambda v
	:= \
	\sum_{\bs \in \Lambda} \Delta^{{\rm Q}}_\bs (v)
	\ = \
	\int_{\RRi} I_\Lambda v (\by)\, \rd \blambda_a(\by).
\end{equation} 
Further, if $\phi \in V'$ is a bounded linear functional on $V$, denote by $\langle \phi, v \rangle$ the value of $\phi$ in $v$.  

We   represent the operator $	Q_{\Lambda(\xi)}$ in the form
\begin{equation*} \label{Q_Lambda_rev(xi)=}
	Q_{\Lambda(\xi)}v			
	\ = \
	\sum_{(\bs,\be,\bk) \in G(\xi)}  (-1)^{|\be|_1} \omega_{\bs - \be;\bk} v(\by_{\bs - \be;\bk}),
\end{equation*}
where	$\omega_{\bs;\bk}:= \prod_{j \in J_\bs} \omega_{s_j,k_j}$.

\begin{theorem}\label{thm:quadrature-X}
Under the assumptions and notations of Theorem~\ref{thm:coll-approx-X}
	there exists a constant $C$ such that for each $n > 1$, we can construct a sequence of points  
	$(\by_{\bs - \be;\bk})_{(\bs,\be,\bk) \in G(\xi_n)}$ so that
	$|G(\xi_n)| \le n$ and
	\begin{equation*} \label{u-Q_Lambda u}
		\left\|\int_{\RRi}u(\by)\, \rd \blambda_a(\by ) - Q_{\Lambda(\xi_n)}u\right\|_V
		\ \le \
		Cn^{-(1/p - 1)},
	\end{equation*}		
	and, if  additionally, $\phi \in V'$ is a bounded linear functional on $V$, 
	\begin{equation*} \label{u-Q_Lambda u-phi}
		\left|\int_{\RRi} \langle \phi, u (\by) \rangle\, \, \rd 
		\blambda_a(\by ) -  \left\langle \phi, Q_{\Lambda(\xi_n)} u \right\rangle \right|
		\ \le \
		C\norm{\phi}{V'} n^{-(1/p - 1)}.
	\end{equation*}	
\end{theorem}
\begin{proof}	
	This theorem can be derived from Theorem~\ref{thm:coll-approx-X} in a manner analogous to the proof of  \cite[Corollary 5.2]{Dung21}. For completeness, we give a detailed proof.
	For a given $n \in \NN$, let $\Lambda(\xi_n)$ be the set defined in Theorem~\ref{thm:coll-approx-X}.
	By definition we have 
	\begin{equation} \label{eq[int_RRi]}
		\int_{\RRi} u(\by)\, \rd \blambda_a(\by ) - Q_{\Lambda(\xi_n)}u 
		\ = \
		\int_{\RRi} \left(u(\by) - I_{\Lambda(\xi_n)}u(\by)\right) \, \rd \blambda_a(\by).
	\end{equation}
	Hence, applying  Theorem~\ref{thm:coll-approx-X}, we obtain (i):
	\begin{equation} \nonumber
		\begin{split}
			\left\|\int_{\RRi} u(\by)\, \rd \blambda_a(\by ) - Q_{\Lambda(\xi_n)}u \right\|_V
			\ &\le \
			\left\|u- I_{\Lambda(\xi_n)}u\right\|_{L_1(\RRi,V;\blambda_a)} 
\\[1.5ex]
&\le 	
	\left\|u- I_{\Lambda(\xi_n)}u\right\|_{L_2(\RRi,V;\blambda_a)} 
			\ \le \
				Cn^{-(1/p - 1)}.
					\end{split}
	\end{equation}
	
	Similarly to \eqref{eq[int_RRi]}, there holds the equality
	\begin{equation} \nonumber
	\Ii:=	\int_{\RRi} \langle \phi, u(\by) \rangle\, \, \rd \blambda_a(\by ) 
		- Q_{\Lambda(\xi_n)} \langle \phi, v(\by) \rangle
		\ = \
		\int_{\RRi} \langle \phi, u(\by)  - I_{\Lambda(\xi_n)}  u(\by) \rangle \, \rd \blambda_a(\by).
	\end{equation}
	Hence, applying Theorem~\ref{thm:coll-approx-X}, we obtain (ii):
	\begin{equation} \nonumber
		\begin{split}
			\left|\Ii\right|
			&\le 
			\int_{\RRi} \left|\langle \phi, u(\by)  - I_{\Lambda(\xi_n)}  u(\by) \rangle \right|\, \rd \blambda_a(\by)
		\le 
			\int_{\RRi} \|\phi\|_{V'} \|u(\by)  - I_{\Lambda(\xi_n)}   u(\by)\|_V \, \rd \blambda_a(\by)
			\\[1.5ex]
			 &\le 
			\|\phi\|_{V'} \|u - I_{\Lambda(\xi_n)}   u\|_{L_1(\RRi,V;\blambda_a)} 
			\le 
			\|\phi\|_{V'} \|u - I_{\Lambda(\xi_n)}   u\|_{L_2(\RRi,V;\blambda_a)} 
			\le 	
			C\norm{\phi}{V'} n^{-(1/p - 1)}.
		\end{split}
	\end{equation}
	\hfill
\end{proof}

\subsection{Extended least squares  approximations}

Let us analyze the convergence rate of the extended least-squares sampling algorithms and associated quadratures for the solution $u(\by)$ to  equation \eqref{SPDE}--\eqref{log-Laplace} based on recent results from \cite{BD2024}.
Let the assumptions of either Theorem~\ref{thm: summability-ell_p-rho-tL} 
or Theorem \ref{thm: summability-ell_p-b-tL} hold for some $0<p<2$. 
Then there holds \eqref{ell_2-summability2} for $u(\by)$.
As shown in the proof Theorem \ref{thm:coll-approx-X},  the series \eqref{series-tL} converges  unconditionally in 
$L_2(\RRi,V;\blambda_a)$ to $u$.  Putting  $U:= \RRi$, $\sigma_j:=\sigma_{\bs_j}$, \ $\varphi_j:=\tL_{\bs_j}$ and $u_j:= \tilde{u}_{\bs_j}$,  
hence,  we can reorder the countable set $\FF$ as $\FF = (\bs_j)_{j \in \NN}$ so that the sequence 
$(\sigma_{\bs_j})_{j \in \NN}$ is non-decreasing, and the weak solution $u$ is represented by the series
\begin{equation*} 
	u = \sum_{j \in \NN}  u_j\,\varphi_j,
\end{equation*}
with
\begin{equation} \label{w-summability}
	\left(\sum_{j \in \NN} (\sigma_j \|u_j\|_V)^2\right)^{1/2} \ \le M^{1/2}, \ \ \text{with} \ \
	\norm{\brac{\sigma_j^{-1}}_{j \in \NN}}{\ell_q(\NN)} \le M^{1/q}.
\end{equation}
Notice that $u(\by)$ is well-defined for every $\by\in U_0$, where $U_0$ is the set defined as 
in \eqref{U_0}.  As noted earlier, the set $U_0$ has full measure, i.e., $\blambda_a(U_0)=1$.

Let us construct an extension of a least squares approximation in the space $L_2(\RRi,\CC;\blambda_a)$ to a space $L_2(\RRi,V;\blambda_a)$.
For $n,m\in\NN$ with $n\ge m$, let $\by_1, \dots, \by_{n}\in U_0$ be points, $\omega_1, \dots, \omega_{n}\ge 0$ be weights, and $V_m = \operatorname{span}\{\varphi_j\}_{j=1}^{m}$ the subspace spanned by the functions $\varphi_j$, $j=1,...,m$.
The weighted least squares approximation $S_{n}^{\CC} g= S_{n}^{\CC}(\by_1, \dots, \by_{n}, \omega_1, \dots, \omega_{n}, V_m) g$ of a function $g\colon \RRi\to\CC$ is given by
\begin{equation} \label{least-squares-sampling1}
	S_{n}^{\CC} g
	= \operatorname{arg\,min}_{\varphi_j\in V_m} \sum_{i=1}^{n} \omega_i |g(\by_i) - \varphi_j(\by_i)|^2 .
\end{equation}
For every $n \in \NN$, let
\begin{equation*} \label{S_nCC}
	S_{n}^\CC g: = \sum_{i=1}^{n} g(\by_i) h_i,
\end{equation*}	
be the least squares sampling algorithm constructed by 
\eqref{least-squares-sampling1}  for these sample points and weights, where 
$h_1,...,h_{n} \in L_2(\RRi,\CC;\blambda_a)$.
Hence we immediately obtain the extension of this least squares algorithm to the Bochner space $L_2(\RRi,V;\blambda_a)$
by replacing $g\in L_2(\RRi,\CC;\blambda_a)$ with $v\in L_2(\RRi,V;\blambda_a)$:
\begin{equation} \label{S_nX}
	S_{n}^V v
	= \sum_{i=1}^{n} v(\by_i) h_i.
\end{equation}
As the least squares approximation is a linear operator, worst-case error bounds carry over from the usual Lebesgue space $L_2(\RRi,\CC;\blambda_a)$ to the Bochner space $L_2(\RRi,X;\blambda_a)$.
We define $	\tilde{S}_n^X$ by
\begin{equation*}\label{tilde{S}_n^X:=}
	\tilde{S}_n^V v
	:= 	
	S_n^V(\by_1, \dots, \by_n, \omega_1, \dots, \omega_n, V_m) v
	:= \sum_{i=1}^{n} v(\by_i) h_i
	, \ \ 
	m:=\Big\lceil \frac{n}{43200}
	\Big\rceil. 
\end{equation*}

\begin{theorem} \label{thm: LS-sampling}
	Let the assumptions of either Theorem~\ref{thm: summability-ell_p-rho-tL} 
	or Theorem \ref{thm: summability-ell_p-b-tL} hold for some $0<p<2$. 
	Then there exists  a constant $C$ such that 
	for any $n\in \NN$, there exist points 
	$\by_1, \dots, \by_n\in U_0$ and weights  $\omega_1, \dots, \omega_n$ such that 
	\begin{itemize}
		\item[{\rm (i)}] 
		\begin{equation*}\label{u - S_{n}^V u-L}
			\norm{u - \tilde{S}_n^V u}{L_2(\RRi,V;\blambda_a)}
			\ \le CM n^{-(1/p -1/2)};
		\end{equation*}	
		\item[{\rm (ii)}] 
		moreover, for the quadrature 
		\begin{equation*} \label{S_nV-L}
			\tilde{Q}_n^V u
			= \sum_{i=1}^n w_i u(\by_i), \quad   w_i:= \int_{\RRip}  h_i(\by) \rd \blambda_a(\by),
		\end{equation*}
		it holds
		\begin{equation} \label{u-Q-LS}
			\left\|\int_{\RRi}u(\by)\, \rd \blambda_a(\by ) - \tilde{Q}_n^V u\right\|_V
			\ \le \
			CM n^{-(1/p -1/2)},
		\end{equation}
		and, if  additionally, $\phi \in V'$ be a bounded linear functional on $V$, 
		\begin{equation} \label{u-Q_Lambda u-phi-L}
			\left|\int_{\RRi} \langle \phi, u (\by) \rangle\, \, \rd 
			\blambda_a(\by ) - \left\langle \phi,  \tilde{Q}_n^V  u \right\rangle \right|
			\ \le \
			C\norm{\phi}{V'} n^{-(1/p -1/2)}.
		\end{equation}
	\end{itemize}			
\end{theorem}	

\begin{proof}
	Claim (i) is directly derived from \cite[Corollary 2.1]{BD2024} and the weighted 
	$\ell_2$-summability~\eqref{w-summability}. Claim (ii) can be obtained from claim (i) in a manner similar to the proof of Theorem~\ref{thm:quadrature-X}. 
	\hfill
\end{proof}

\begin{remark} \label{remark4.2}
	{\rm
		The convergence rate of the approximation of $u$ by least squares sampling algorithms  (Theorem~\ref{thm: LS-sampling}) is markedly superior to that of the approximation of $u$ by sparse-grid polynomial interpolations (Theorem~\ref{thm:coll-approx-X}) by a factor of $n^{-1/2}$. As noted Remark~\ref{remark4.1}, the former matches the convergence rate of the approximation by finite truncations of the Laguerre GPWPC expansion and of the optimal  convergence rate of best $n$-term Laguerre piecewise-polynomial approximation of $u$.
	}	
\end{remark}

 \medskip
\noindent
{\bf Acknowledgments:} 
The work  is funded by the Vietnam National Foundation for Science and Technology Development (NAFOSTED) under the Vietnamese-Swiss Joint Research Project, Grant 
No. IZVSZ2$_{ - }$229568.
A part of this work was done when the author was working at the Vietnam Institute for Advanced Study in Mathematics (VIASM). They would like to thank the VIASM for providing a fruitful research environment and working condition.

\bibliographystyle{abbrv}
\bibliography{Laplace-SPDE}

\end{document}